\documentclass[10pt,a4paper]{article}
\usepackage{amsmath}
\usepackage{amsfonts}
\usepackage[latin1]{inputenc}
\usepackage{amsmath}
\usepackage{amsthm}
\usepackage{amsfonts}
\usepackage{amssymb}

\newtheorem{thm}{Theorem}[section]
\newtheorem{co}[thm]{Corollary}

\newtheorem{lem}[thm]{Lemma}
\numberwithin{equation}{section}

\renewcommand{\P}{\mathbb{P}}
\newcommand{\E}{\mathbb{E}}

\newcommand{\T}{\mathbb{T}}

\begin{document}
\title{Moments for  multi-dimensional Mandelbrot's cascades }
\author{Chunmao HUANG
\footnote{
Email addresses: cmhuang@hitwh.edu.cn (C. Huang) .}
\\
\small{\emph {Harbin institute of technology at Weihai, Department of mathematics, 264209, Weihai, China}}}

 \date{}

\maketitle

\begin{abstract}
We  consider the distributional equation
$\textbf{Z}\stackrel{d}{=}\sum_{k=1}^N\textbf{A}_k\textbf{Z}(k)
$,
where $N$ is a random variable taking value in $\mathbb N_0=\{0,1,\cdots\}$,
$\textbf{A}_1,\textbf{A}_2,\cdots$ are $p\times p$ non-negative random matrix, and
$\textbf{Z},\textbf{Z}(1),\textbf{Z}(2),\cdots$ are $i.i.d$ random
vectors in in $\mathbb{R}_+^p$ with  $\mathbb{R}_+=[0,\infty)$, which are
independent of $(N,\textbf{A}_1,\textbf{A}_2,\cdots)$.
Let $\{\mathbf Y_n\}$ be the multi-dimensional Mandelbrot's martingale defined as sums of products of random matrixes indexed by nodes of a Galton-Watson tree plus an appropriate vector. Its limit $\mathbf Y$ is a solution of the equation above. For $\alpha>1$, we show respectively a sufficient condition and a necessary condition  for $\E\|\mathbf Y\|^\alpha\in(0,\infty)$. Then for a non-degenerate solution $\mathbf Z$ of the equation above, we show the decay rates of $\E e^{-\mathbf t\cdot \mathbf Z}$ as $\|\mathbf t\|\rightarrow\infty$ and those of the tail probability $\mathbb P(\mathbf y\cdot \mathbf Z\leq x)$ as $x\rightarrow 0$ for given $\mathbf y=(y^1,\cdots,y^p)\in \mathbb R_{+}^p$,  and the existence of the harmonic moments of $\mathbf y\cdot \mathbf Z$. As application, these above results  about the moments (of positive and negative orders) of $\mathbf Y$ are applied to a special multitype branching random walk. Moreover, for the case where all the vectors and matrixes of the equation above are complex,
a sufficient condition for the $L^\alpha$ convergence and  the $\alpha$th-moment of the Mandelbrot's martingale $\{\mathbf Y_n\}$ is also established.\\



\end{abstract}



\section {Introduction}
We consider a multi-dimensional  Mandelbrot's martingale $\{\mathbf Y_{n}\}$ defined as sums of products of random matrixes (weights) indexed by nodes of a Galton-Watson tree plus an appropriate vector. We are interested in the existence of the moments of positive and negative  orders of its limit $\mathbf Y$. For the one-dimensional case,  the classical model of Mandelbrot \cite{Mandelbrot74} corresponds to the case where the tree is a fixed $r$-ary tree ($r\geq 2$ being a constant), and all the weights are one-dimensional random variables. This classical model and its variations were studied by many authors in different contexts, see for example:  Bingham \& Doney \cite{Bingham74, Bingham75} for branching processes and general age-dependent branching processes;   Kahane \& Peyri\`{e}re \cite{Kahane76}, Guivarc'h \cite{Guivarch90} and Barral \cite{Barral99} for  multiplicative cascades;   Biggins \cite{Biggins77}  and Biggins \& Kyprianou \cite{Biggins97} for branching random walks;   Durrett \& Liggett \cite{Durrett83} for some infinite particle systems;  R\"osler \cite{Rosler92} for the Quicksort algorithm. A general one-dimensional model (called Mandelbrot's cascades) which  unifies the study of cascades  and  branching random walks  was presented by Liu \cite{liu2000}, where a number of applications were shown. The model considered here is a generalization of the model presented in \cite{liu2000} to the multi-dimensional case. Similar to the  one-dimensional case, our model is also  corresponding to  multi-type branching random walks which  attract some authors' attention recently, see for example Kyprianou \& Rahimzadeh Sani \cite{ks}, Biggins \& Rahimzadeh Sani \cite{bs} and Biggins \cite{b12}. This paper is our first exploration to  multi-dimensional Mandelbrot's cascades. Considering the practicability, we choose to begin with the existence of the moments of $\mathbf Y$, which are useful to study the asymptotic properties   of $\{\mathbf Y_{n}\}$.

\bigskip

Let's present our model and problems. We consider the distributional equation of $\textbf{Z}$:
\begin{equation}\tag{E}\label{E}
\textbf{Z}\stackrel{d}{=}\sum_{k=1}^N\textbf{A}_k\textbf{Z}(k),
\end{equation}
where $N$ is a random variable taking value in $\mathbb N_0=\{0,1,\cdots\}$,
$\textbf{A}_1,\textbf{A}_2,\cdots$ are $p\times p$ non-negative random matrix ;
$\textbf{Z},\textbf{Z}(1),\textbf{Z}(2),\cdots$, which are
independent of $(N,\textbf{A}_1,\textbf{A}_2,\cdots)$, are $i.i.d$ random
vectors in $\mathbb{R}_+^p$ with  $\mathbb{R}_+=[0,\infty)$.

We say a matrix $\mathbf{A}$ is \emph{finite} if all entries of $\mathbf A$ are finite, and
say $\mathbf A$ is \emph{strictly positive} if for some positive integer $n$, all
entries of $\mathbf A^n$ are positive. When a matrix $\mathbf A$ is finite
and strictly positive, the Perron-Frobeninius theorem shows that $\mathbf A$
has a positive maximal eigenvalue $\rho$ and has associated
positive right and left eigenvectors
$\mathbf{v}=(v_1,\cdots,v_p)$ and
$\mathbf{u}=(u_1,\cdots,u_p)$. Moreover, $\mathbf{u}$,
$\mathbf{v}$ can be normalized so that
$\sum\limits_{i=1}^{p}u_i=\sum\limits_{i=1}^{p}u_iv_i=1$. Throughout this paper,
we assume  that\\
\\*
\textbf{Assumption (H)}. The matrix \emph{$\textbf{M}:=\mathbb E\sum\limits_{k=1}^N\textbf{A}_k$ is finite and strictly
positive with the maximum-modulus eigenvalue $1$ and the
corresponding left and right eigenvectors $\textbf{U}=(U_1,\cdots,U_p),\textbf{V}=(V_1,\cdots,V_p)$ normalized such
that $\sum\limits_{i=1}^pU_i=\sum\limits_{i=1}^pU_iV_i=1$.}
\\*

We are interested in the existence of the
 solution with $\alpha$th-moment ($\alpha>1$) of  the equation (\ref{E}), and furthermore, the existence of its harmonic moments.
It is clear that there exists a solution of  equation (\ref{E}). In fact, we can construct a solution (denoted by $\mathbf Y$) in the following way. 
Let $\mathbb N=\{1,2, \cdots\}$ and write
\begin{equation*}
I=\bigcup_{n=0}^{\infty }\mathbb{N}^{n}
\end{equation*}
for the set of all finite sequences $u=u_1\cdots u_n$ with $u_i\in \mathbb N$, where by convention ${\mathbb{N}}^{0}=\{\emptyset \}$ contains the null sequence $\emptyset $. If $u=u_1\cdots u_n\in I$, we write $|u|=n$ for the length of $u$; if  $u=u_1\cdots u_n, v=v_1\cdots v_m\in I$, we write $uv=u=u_1\cdots u_n v_1\cdots v_m$ for the sequence obtained by juxtaposition. In particular, $u\emptyset=\emptyset u=u$. We partially order $I$ by writing $u\le v$ to mean that for some $u'\in I$, $v=uu'$, and by writing $u<v$ to mean that $u\le v$ and $u\ne v$.

Let $\{(N_u,\textbf{A}_{u1},\textbf{A}_{u2},\cdots)\}$ be a family
of independent copies of $(N,\textbf{A}_1,\textbf{A}_2,\cdots)$,
indexed by all the finite sequence $u\in I$. For simplicity, we write $(N, \mathbf A_{1}, \mathbf A_{2},\cdots )$ for $(N_{\emptyset
},\textbf{A}_{\emptyset 1},\textbf{A}_{\emptyset 2},\cdots )$. Let $\mathbb{T}$ be the Galton-Watson tree with
defining elements ($N_{u}$) ($u\in I$): (i) $\emptyset \in \mathbb{T}$; (ii) if $u\in \mathbb{T}$, then $uk\in \mathbb{T}$ if and only if $1\leq k\leq N_{u}$; (iii) if $u k\in \mathbb{T}$, then $u\in \mathbb{T}$.  Here the null sequence $\emptyset $ is the root of the tree $\mathbb{T}$, which can be regarded as
the initial particle; $u k$ represents the $k$-th child of $u$; $N_{u}$ represents the number of offspring of the particle $u$.

Each node of the tree $\mathbb{T}$ is marked with the random vector $(N_u,\mathbf{A}_{u1},\textbf{A}_{u2},\cdots)$. We can imagine that the random matrix $\textbf{A}_{u k}$ is  the  "weight" associated with the edge $(u,uk)$ linking the nodes $u$ and $u k$ if $u\in \mathbb{T}$ and $1\leq k\leq N_{u}$; the values $\textbf{A}_{u k}$ for $k>N_{u}$ are of no influence for our purpose, and will be taken as $0$ for convenience.

Let $\mathbb{T}_{n}=\{u\in\T: |u|=n\}$ be the set
of sequence $u$ in $\mathbb{T}$ with length $|u|=n$. Put
\begin{equation*}
 \mathbf X_{u}=\mathbf A_{u_{1}}\mathbf A_{u_{1}u_{2}}\cdots
\mathbf A_{u_{1}\cdots u_{n}}\quad \mathrm{ if }\
u=u_{1} \ldots u_{n}\in I \quad \mathrm{ for } \;  n\geq 1,
\end{equation*}
and define
\begin{equation}\label{beqYn}
\mathbf Y_{0}=\mathbf V  \quad \mbox{ and }\quad  \mathbf Y_{n}=\sum_{u\in \mathbb{T}_{n}} \mathbf X_{u}\mathbf V  \quad \mathrm{ for } \;
n\geq 1.
\end{equation}
It is not difficult to  verify  that $\textbf{Y}_n=(Y_{n,1},\cdots,Y_{n,p})$ is a non-negative martingale with respect to the filtration
$$\mathcal {F}_n=\sigma((N_{u},\mathbf A_{u1},\mathbf A_{u2},\cdots ):|u|<n),$$
the $\sigma$-field that contains
all information up to generation $n$. We call $\{\textbf{Y}_n\}$ {\em Multi-dimensional Mandelbrot's martingale}. It reduce to the classical Mandelbrot's martingale when the dimension $p=1$. Clearly, there exists a non-negative random vector $\mathbf Y=(Y_1, \cdots, Y_p)\in \mathbb R_+^p$ such that
$$\mathbf Y=\lim_{n\rightarrow \infty }\mathbf Y_{n}$$
almost surely (a.s.) with $\mathbb EY_i\leq V_i$ for all $1\leq i\leq p$ by Fatou's lemma. Notice that
\begin{eqnarray}\label{MCE1}
\mathbf Y_n
=\sum_{u\in \mathbb{T}_{1}}\mathbf A_{u}\mathbf Y_{n-1}(u),
\end{eqnarray}
where  $\{\mathbf Y_n(u)\}$ ($u\in\T_k$) are independent copies of $\mathbf Y_n$ and they are independent of $\mathcal F_k$. Denote $\mathbf Y(u)=\lim\limits_{n\rightarrow\infty}\mathbf Y_n(u)$. Letting $n\rightarrow\infty$ in (\ref{MCE1}), we have
\begin{equation}\label{MCE2}
\textbf{Y}{=}\sum_{k=1}^N\textbf{A}_k\textbf{Y}(k),
\end{equation}
which means that $\mathbf Y$ is a solution of the
equation (\ref{E}).

\bigskip

\paragraph{Example 1.1 Multitype branching random walk (MBRW) } A multitype branching random walk (MBRW) with $p$ types
defined as follows. A single particle $\emptyset$, of type
$i\in\{1,2,\cdots, p\}$ is located at the origin of real line $\mathbb{R}$. It
gives birth to children of the first generation, which are scattered
on $\mathbb{R}$, according to a vector point process
$\mathbf{L}_i=(L_{i1},L_{i2},\cdots,L_{ip})$, where $L_{ij}$ is the
point process counting the number of particles of type
$j\in\{1,2,\cdots, p\}$ born to the particle of type $i$. These
particles of the first generation  reproduce particles to form the
second generation. The displacements of the offsprings of a particle
of type $j$, relative to their parent's position, are given by the
point process $\mathbf{L}_j$. These particles of the second
generation reproduce children to form the next generation, and so
on. All particles behave independently. We denote the position of a
particle $u$ by $S_u$ and the type of $u$ by $\tau (u)$ , then the position of $uk$, the $k$-th child
of $u$ satisfies
$$S_{uk}=S_u+l_{uk},$$
where $l_{uk}$ denotes the displacement of $uk$ relative to $u$ whose distribution is determined by $L_{\tau(u)\tau(uk)}$.

Assume that $N_i:=\sum\limits_{j=1}^pZ_{ij}(\mathbb R)$ has the same distribution for all $1\leq i\leq p$, which means that all particles produce offspring according to the same distribution if we don't care the type. Under this assumption,  all particles $u\in I$ associated with the numbers of their offspring $N_u$ form a Galton-Watson tree $\mathbb T$ described above. We remark that this assumption is not necessary in a usual MBRW, so the example presented here is just a special case of MBRW. For more information and more results about the usual MBRW, cf \cite {bs, b12, ks}.

For $t\in \mathbb{R}$, define the matrix $\tilde {\mathbf  M}(t)=(\tilde M_{ij}(t))$ as
$$
\tilde M_{ij}(t)=\sum_{\substack{u\in\mathbb T_1\\\tau(u)=j}}e^{-tS_u}\qquad\text{( $\tau(\emptyset)=i$).}
$$
Assume that  $\tilde{\mathbf M}(t)$ defined above is finite
and strictly positive.  Denote the positive maximal eigenvalue  of $\tilde{\mathbf M}(t)$ by $\tilde \rho(t)$ and the associated normalized positive left and right eigenvectors by $\tilde{\mathbf U}(t)=(\tilde U_1(t),\cdots, \tilde U_p(t))$ and $\tilde{\mathbf V}(t)=\tilde V_1(t),\cdots, \tilde V_p(t))$ respectively. For each $i=1,2,\cdots, p$, let
\begin{equation*}
W_{n,i}(t):=\frac{\sum\limits_{u\in\T_n}\tilde V_{\tau(u)}(t)e^{-tS_u}}{\tilde V_i(t)\tilde \rho(t)^n}\qquad (\tau(\emptyset)=i).
\end{equation*}
It is known that for each $i=1,2,\cdots, p$, $\{W_{n,i}(t)\}$ forms a non-negative martingale with mean one, hence it converges a.s. to a non-negative random variable $W^i(t)$ with $\mathbb EW_i(t)\leq1$. Write
$$\mathbf Y_n=\left(W_{n,1}(t)\tilde V_1(t),\;\;W_{n,2}(t)\tilde V_2(t),\;\;\cdots\;\;, W_{n,p}(t)\tilde V_p(t)\right).$$
We can see that the martingale $\{\mathbf Y_n\}$ is just the Mandelbrot's martingale defined in (\ref{beqYn}) if we put the random matrix $\mathbf A_k=((\mathbf A_k)_{ij})$, where
$$(\mathbf A_k)_{ij}=\frac{e^{-tS_k}}{\tilde \rho(t)}\mathbf{1}_{\{\tau(k)=j\}}\qquad (\tau(\emptyset)=i).$$
Indeed, with $\mathbf A_k$, we have $\mathbf M=\frac{\tilde{\mathbf M}(t)}{\tilde \rho(t)}$, so that $\mathbf V=\tilde{\mathbf V}(t)$. Notice that for $u=u_1\cdots u_n$,
$$(\mathbf A_{u_{1}\cdots u_{n}})_{ij}=\frac{e^{-tS_u}}{\tilde \rho(t)^n}\mathbf{1}_{\{\tau(u)=j\}}\qquad (\tau(\emptyset)=i).$$
Thus by (\ref{beqYn}), for each $i=1,2,\cdots, p$, with $\tau(\emptyset)=i$,
$$\mathbf Y_{n,i}=\sum_{u\in \mathbb{T}_{n}}\frac{e^{-tS_u}}{\tilde \rho(t)^n}\tilde V_{\tau(u)}(t)
=W_{n,i}(t)\tilde V_i(t)$$
Therefore, the limit of $\mathbf Y_n$, namely,
$$\mathbf Y=\left(W_1(t)\tilde V_1(t),\;\;W_2(t)\tilde V_2(t),\;\;\cdots\;\;, W_p(t)\tilde V_p(t)\right)$$
satisfies (\ref{MCE2}).


\section{Main results}

Let $\mathbf Y$ be the limit of the Mandelbrot's martingale $\{\mathbf Y_n\}$.
We first discuss the existence of the
$\alpha$th-moment ($\alpha>1$) of $\textbf{Y}$, which implies its non-degeneracy.

For $t\in\mathbb{R}$ fixed,
define the random matrix $\textbf{A}_k^{(t)}=((\textbf{A}_k^{(t)})_{ij})$ as $(\textbf{A}_k^{(t)})_{ij}:=[(\textbf{A}_k)_{ij}]^t$. Let
$$\mathbf M{(t)}:=\mathbb E \sum_{k=1}^N \mathbf A_k^{(t)}.$$
When $\mathbf M{(t)}$ is finite
and strictly positive, we denote the its positive maximal eigenvalue  by $ \rho(t)$ and the corresponding positive left and right eigenvectors by ${\mathbf U}(t)=( U_1(t),\cdots,  U_p(t))$ and ${\mathbf V}(t)= (V_1(t),\cdots,  V_p(t))$ normalized such
that $\sum\limits_{i=1}^pU_i(t)=\sum\limits_{i=1}^pU_i(t)V_i(t)=1$. Define
\begin{equation*}
 \mathbf X_{u}^{(t)}=\mathbf A_{u_{1}}^{(t)}\mathbf A_{u_{1}u_{2}}^{(t)}\cdots
\mathbf A_{u_{1}\cdots u_{n}}^{(t)}\quad \mathrm{ if }\
u=u_{1} \ldots u_{n}\in I \quad \mathrm{ for } \;  n\geq 1,
\end{equation*}
\begin{equation}\label{beqYn1}
\mathbf Y_{0}^{(t)}=\mathbf V(t) \quad \mbox{ and }\quad  \mathbf Y_{n}^{(t)}=\sum_{u\in \mathbb{T}_{n}} \mathbf X_{u}^{(t)}\mathbf V(t)\quad \mathrm{ for } \;
n\geq 1.
\end{equation}
Clearly, $\textbf{Y}_n^{(t)}=(Y_{n,1}^{(t)},\cdots,Y_{n,p}^{(t)})$ is  a non-negative martingale with mean $\textbf{V}(t)$,
so it converges a.s. to a random vector $\textbf{Y}^{(t)}=(Y^{(t)}_1,\cdots,Y^{(t)}_p)$. In particular, when $t=1$,
we have $\textbf{X}_u^{(1)}=\textbf{X}_u$, $\rho(1)=1$ and $\textbf{V}(1)=\textbf{V}$, hence $\textbf{Y}_n^{(1)}=\textbf{Y}_n$
and $\mathbf Y^{(1)}=\mathbf Y$.

Further more, define the matrix
$\textbf{M}_n(t)=((\textbf{M}_n(t))_{ij})$ as $$(\textbf{M}_n(t))_{ij}:=\mathbb E\sum_{u\in \T_n}[(\textbf{X}_u)_{ij}]^t$$
with the maximum-modulus eigenvalue denoted by $\rho_n(t)$ and the
corresponding normalized positive left and right eigenvectors by $\textbf{U}_n(t)=(U_{n,1}(t)\cdots,U_{n,p}(t))$ and $\textbf{V}_n(t)=(V_{n,1}(t),\cdots,V_{n,p}(t))$.
 In particular, $\rho_1(t)=\rho(t)$.

\bigskip

We declare that throughout this paper the notation norm $\|\mathbf A\|$ represents any one of the matrix norms if $\mathbf A$ is a matrix, and $\|\mathbf u\|=\sum\limits_{j=1}^p|u_j|$  is the $L^1$-norm of $\mathbf u=(u_1, \cdots, u_p)$ if $\mathbf u$ is a vector.

\begin{thm}[Moments]\label{MCT1}
 Let $\alpha>1$.
 \begin{itemize}

 \item[(a)]If $\mathbb E\|\sum\limits_{k=1}^N\textbf{A}_k\|^\alpha<\infty$ and $p^{(\alpha-1)}\rho_n(\alpha)<1$
for some positive integer $n$, then
$$0<\mathbb E\|\mathbf Y\|^\alpha<\infty\qquad \text{and}\qquad \mathbb E\mathbf Y=\mathbf V.$$

\item[(b)]
 Conversely, if $0<\mathbb E\|\mathbf Y\|^\alpha<\infty$, then $\mathbb E\|\sum\limits_{k=1}^N\textbf{A}_k\|^\alpha<\infty$ and $\rho_n(\alpha)\leq 1$ for all $n$. If additionally
\begin{equation}\label{MCEC2}
\text{$\mathbb P(\text{$\forall k\in\{1, 2,\cdots, N\}$, $\mathbf A_k$ has a positive column vector})>0$,}
\end{equation}
then  $\rho_n(\alpha)<1$ for all $n$.

\end{itemize}
\end{thm}

\noindent\textbf{Remark 2.1. } (i) For $\alpha>1$,  under Assumption (H), the condition  $\mathbb E\|\sum\limits_{k=1}^N\textbf{A}_k\|^\alpha<\infty$ ensures that $\mathbf M(\alpha)$ is finite and strictly positive, so that $\rho(\alpha)$ exists. Notice that for each $t\in \mathbb R$ fixed,
$$[\textbf{M}(t)]^n\leq\textbf{ M}_n(t)\leq p^{(\alpha-1)(n-1)}[\textbf{M}(t)]^n,$$
where for two matrix  $\mathbf A=(a_{ij}), \mathbf B=(b_{ij})$, the inequality $\mathbf A\leq \mathbf B$ means that $a_{ij}\leq b_{ij}$ for all $i, j$.
Thus the existences of $\rho(t)$  and $\rho_n(t)$ are equivalent, and  we moreover have for each $t\in \mathbb R$ fixed,
$$\text{$\rho(t)^n\leq \rho_n(t)\leq p^{(\alpha-1)(n-1)}\rho(t)^n$.}$$
Therefore, under Assumption (H) and the condition  $\mathbb E\|\sum\limits_{k=1}^N\textbf{A}_k\|^\alpha<\infty$,
$\rho_n(t)$ exists for all $t\in [1,\alpha]$ and for all $n$. Besides, we remark that  the condition  $\mathbb E\|\sum\limits_{k=1}^N\textbf{A}_k\|^\alpha<\infty$ is equivalent to $\mathbb E \|\mathbf Y_1\|^\alpha<\infty$.

(ii) Under Assumption (H), $\mathbb E\|\mathbf Y\|^\alpha>0$  is equivalent to $\mathbb E(Y_i)^\alpha>0$ for all $i\in\{1,\cdots,p\}$. Indeed, by (\ref{MCE2}), one can see that $\mathbb E \mathbf Y$ is a an eigenvector associated to the eigenvalue $1$. If it is non-trivial, i.e. $\mathbb E \mathbf Y\neq \mathbf 0$, then $\mathbb E \mathbf Y=c\mathbf V$ for some constant $c> 0$, which implies that $\mathbb E \mathbf Y$ is positive.

\bigskip
\bigskip

Theorem \ref{MCT1}(a) shows a sufficient condition for the existence of the
$\alpha$th-moment ($\alpha>1$) of $\textbf{Y}$, or equivalently, the $L^\alpha$ convergence of the martingale $\{\mathbf Y_n\}$ to its limit $\textbf{Y}$. If $\mathbb E(Y_i)^\alpha<\infty$, it is obvious that $\mathbb E Y_i=V_i$ and $\mathbb P(Y_i>0)>0$. As $\mathbf Y$ is a solution of the equation (\ref{E}), Theorem \ref{MCT1}(a) in fact also gives the existence of a non-trivial solution of equation (\ref{E}).

Moreover, if $p^{(\alpha-1)}\rho_n(\alpha)<1$
for some positive integer $n$, Theorem \ref{MCT1} implies that $0<\mathbb E \|\mathbf Y\|^\alpha<\infty$ if and only if $\mathbb E \|\mathbf Y_1\|^\alpha<\infty$, which reveals that $\mathbf Y_{1}$ and $\mathbf Y$ would have the same asymptotic
properties.
In particular, for $p=1$, if $\mathbb P(\forall k\in\{1, 2,\cdots, N\}, A_k>0 )>0$, Theorem \ref{MCT1} says that
$$ 0<\mathbb E Y^\alpha<\infty\qquad \text{if and only if}\qquad\mathbb EY_1^\alpha<\infty \;\;\text{and}\;\; \rho(\alpha)<1.$$ This result was obtained by Liu (\cite{liu2000}, Theorem 2.1) with the help of a size-biased measure. Here our proof will present a different idea based on  inequalities for martingale. Our method, which is available for both $p=1$ and $p>1$, also avoids the trouble of finding an convenient size-biased measure for the case where $p>1$. We mention that this method  can also be used to the  complex case where $\textbf{A}_k$ are complex random matrixes and $\textbf{Z}$ and $\textbf{Z}(k)$  are complex random vectors, see Section 6.

\bigskip

\bigskip

Now we consider the existence of harmonic moments of $\textbf{Y}$, i.e., $\mathbb E (Y_i)^{-\lambda}<\infty$, for each $i\in\{1,2,\cdots, p\}$, where $\lambda>0$. We shall deal with a more general case, with a general non-trivial solution of equation (\ref{E}), denoted still by $\mathbf Z$, instead of $\mathbf Y$.

Let $\mathbf Z$ be a non-trivial solution of equation (\ref{E}). Then we have $\mathbb P(\mathbf Z>\mathbf 0)>0$, where $\mathbf Z>\mathbf 0$ means that $Z_i>0$ for all $i=1,2,\cdots, p$. Assume that (\ref{MCEC2}) holds, and
\begin{equation}\label{MCA1}
\mathbb P(N=0)=0,\qquad \mathbb P(N=1)<1.
\end{equation}
In fact, assumption (\ref{MCEC2}) is object to ensure that the probability $\mathbb P(\mathbf Z=\mathbf 0)$ is a solution of the equation $f(q)=q$, where $f(s)=\mathbb E s^ N$ ($0\leq s\leq 1$) is the generating function of $N$. Since $\mathbb P(\mathbf Z=\mathbf 0)<1$, under assumptions (\ref{MCA1}), by the unity of solution, we have $\mathbb P(\mathbf Z=\mathbf 0)=0$, or namely, $\mathbb P(\mathbf Z>\mathbf 0)=1$. Let
\begin{equation}
\phi(\mathbf t)=\mathbb Ee^{-\mathbf t\cdot\mathbf Z},\qquad\mathbf t=(t_1,\cdots,t_p)\in \mathbb R_{+}^p,
\end{equation}
be the Laplace transform of $\mathbf Z$, where  we write $\mathbf u\cdot\mathbf v=\sum\limits_{j=1}^pu_jv_j$ for the inner product of two vectors $\mathbf u$ and $\mathbf v$.
We are interested in the decay rate of $\phi(\mathbf t)$ as $\|\mathbf t\|\rightarrow\infty$ and that of the tail probability $\mathbb P(\mathbf y\cdot \mathbf Z\leq x)$ as $x\rightarrow0$, for given $\mathbf y=(y_1,\cdots,y_p)\in \mathbb R_{+}^p$ , as well as the harmonic moment $\mathbb E (\mathbf y\cdot \mathbf Z)^{-\lambda}$ for $\lambda>0$. Set
$$\underline {m}:=\mbox{\emph{essinf}} \;N$$
Then $\underline {m}\geq1$, since $\mathbb P(N=0)=0$. We have the following result.

\begin{thm}[Harmonic moments]\label{MCT2}
Assume (\ref{MCEC2}) and (\ref{MCA1}). Write $a_{ij}=(\mathbf A_1)_{ij}$.  If
$$\mathbb E \left(\min\limits_i\sum\limits_{j=1}^p a_{ij}\right)^{-\lambda}<\infty\qquad \text{and}\qquad\mathbb E \left[\left(\min\limits_i\sum\limits_{j=1}^p a_{ij}\right)^{-\lambda}\mathbf 1_{\{N=1\}}\right]<1
$$
for some $\lambda>0$, then
$$\phi(\mathbf t)=O(\|\mathbf t\|^{-\lambda})\quad(\|\mathbf t\|\rightarrow\infty),$$
and for every fixed non-zero $\mathbf y=(y_1,\cdots,y_p)\in \mathbb R_{+}^p$,
$$\mathbb P(\mathbf y\cdot \mathbf Z\leq x)=O(x^{\lambda})\quad(x\rightarrow0),\qquad\quad \mathbb E (\mathbf y\cdot \mathbf Z)^{-\lambda_1}<\infty\quad(0<\lambda_1<\lambda). $$
If additionally $\underline {m}>1$ and $\mathbb E\left[\prod\limits_{k=1}^{\underline {m}}\left(\min\limits_i\sum\limits_{j=1}^p (\mathbf A_k)_{ij}\right)^{-\lambda}\right]<\infty$, then
$$\phi(\mathbf t)=O(\|\mathbf t\|^{-\underline {m}\lambda})\;\;(\|\mathbf t\|\rightarrow\infty),\quad \mathbb P(\mathbf y\cdot \mathbf Z\leq x)=O(x^{\underline {m}\lambda})\;\;(x\rightarrow0),\quad \mathbb E (\mathbf y\cdot \mathbf Z)^{-\underline {m}\lambda_1}<\infty\;\;(0<\lambda_1<\lambda).$$
\end{thm}

\bigskip

From Theorem \ref{MCT2}, we can deduce similar results for each component $Z_i$ of $\mathbf Z$. Let $\phi_i(t)=\mathbb Ee^{-tZ_i}$ ($t>0$) be the Laplace transform of $Z_i$. Denote by $\mathbf e_i$ the vector which the $i$-th  component is $1$ and the others are $0$. Then $\phi_i(t)=\phi(t\mathbf e_i)$, and $\mathbf e_i\cdot \mathbf Z=Z_i$. Applying Theorem \ref{MCT2} to $\phi(t\mathbf e_i)$ and $\mathbf e_i\cdot \mathbf Z$, we immediately get the following corollary.

\begin{co}\label{MCC1}
Under the conditions of Theorem \ref{MCT2}, we have for each $i\in\{1,2,\cdots,p\}$,
$$\phi_i( t)=O( t^{-\lambda})\;\;( t\rightarrow\infty),\quad \mathbb P(Z_i\leq x)=O(x^{\lambda})\;\;(x\rightarrow0),\quad \mathbb E (Z_i)^{-\lambda_1}<\infty\;\;(0<\lambda_1<\lambda).$$
If additionally $\underline {m}>1$ and $\mathbb E\left[\prod\limits_{k=1}^{\underline {m}}\left(\min\limits_i\sum\limits_{j=1}^p (\mathbf A_k)_{ij}\right)^{-\lambda}\right]<\infty$, then
$$\phi_i( t)=O( t^{-\underline {m}\lambda})\;\;( t\rightarrow\infty),\quad \mathbb P(Z_i\leq x)=O(x^{\underline {m}\lambda})\;\;(x\rightarrow0),\quad \mathbb E (Z_i)^{-\underline {m}\lambda_1}<\infty\;\;(0<\lambda_1<\lambda).$$
\end{co}

\bigskip

For $p=1$, Theorem \ref{MCT2} (or Corollary \ref{MCC1}) coincides with  the results of Liu (\cite{liu2}, Theorems 2.1 and 2.4). But when $p>1$, to find the critical value for the existence of harmonic moments like \cite{liu2} seems difficult.
Similar to  (\cite{liu2}, Theorem 2.5), we also have result below about the exponential decay rate of $\phi(\mathbf t)$.

\begin{thm}[The exponential case]\label{MCT3}
Assume that (\ref{MCEC2}) holds, $\underline{m}\geq 2$ and $\min\limits_{i,j}\mathbf (A_k)_{ij}\geq \underline {a}$ a.s. for some constant $\underline {a}>0$ and all $1\leq k\leq \underline {m}$.
\begin{itemize}
\item[(a)] If $\mathbb P\left(N=\underline {m}\right)>0$, then there exists a constant $C_1>0$ such that for all $\|\mathbf t\|>0$ large enough,
    $$\phi(\mathbf t)\leq \exp\{-C_1\|\mathbf t\|^\gamma\},$$
    and for every fixed non-zero $\mathbf y=(y_1,\cdots,y_p)\in \mathbb R_{+}^p$, there exists a constant $C_{1,\mathbf y}>0$ such that for all $x>0$ small enough,
    $$\mathbb P(\mathbf y\cdot \mathbf Z\leq x)\leq  \exp\{-C_{1,\mathbf y}x^{-\gamma/(1-\gamma)}\},$$
    where $\gamma=-\log\underline{m}/\log{(\underline{a}p)}\in(0,1)$.
\item[(b)]For some $\varepsilon>0$ satisfying $(\underline{a}+\varepsilon)p\underline {m}<1$, if $\mathbb P\left(N=\underline {m},\; \max\limits_{ij}(A_k)_{ij}\leq \underline{a}+\varepsilon\; \text{for all } 1\leq k\leq \underline {m}\right)>0$, then there exists a constant $C_2>0$ such that for all $\|\mathbf t\|>0$ large enough,
    $$\phi(\mathbf t)\geq \exp\{-C_2\|\mathbf t\|^{\gamma(\varepsilon)}\},$$
    and for every fixed $\mathbf y=(y_1,\cdots,y_p)\in \mathbb R_{+}^p$, there exists a constant $C_{2,\mathbf y}>0$ such that for all $x>0$ small enough,
    $$\mathbb P(\mathbf y\cdot \mathbf Z\leq x)\geq  \exp\{-C_{2,\mathbf y}x^{-\gamma(\varepsilon)/(1-\gamma(\varepsilon))}\},$$
    where $\gamma(\varepsilon)=-\log\underline{m}/\log{[(\underline{a}+\varepsilon) p]}\in(0,1)$.
\end{itemize}
\end{thm}

\bigskip

Finally, as applications of the above moment results for the limit $\mathbf Y$ of Mandelbrot's martingale $\{\mathbf Y_n\}$, we consider the MBRW described in Example 1.1 and show the sufficient conditions for the existence of moments (of positive and negative orders) of $W_i(t)$, for each $i=1, 2, \cdots, p$ and for $t\in\mathbb R$ fixed. For MBRW, it is obvious that (\ref{MCEC2}) is satisfied. Notice that $\mathbf{M}(\alpha)=\frac{\tilde{\mathbf M}(\alpha t)}{\tilde \rho(t)^\alpha}$, which leads to  $\rho(\alpha)=\frac{\tilde{\rho}(\alpha t)}{\tilde \rho(t)^\alpha}$. Applying Theorem \ref{MCT1} yields the result for the moments of positive orders, and   Theorem \ref{MCT2} yields the one for the  moments of negative orders.

\begin{co}[Application to MBRW]\label{MCC2}We consider the MBRW described in Example 1.1.
\begin{itemize}
\item[(a)]Let $\alpha>1$. If $\max\limits_{i}\mathbb E\left(W_{1,i}(t)\right)^\alpha<\infty$ and $p^{\alpha-1}\frac{\tilde{\rho}(\alpha t)}{\tilde \rho(t)^\alpha}<1$, then $\max\limits_{i}\mathbb E [W_i(t)]^\alpha<\infty$.
\item[(b)]Assume (\ref{MCA1}). Denote by $S_1^i$ the displacement of the first child ${1}$ of the initial particle $\emptyset$ of type $i\in\{1,\cdots, p\}$. Let $\lambda>0$. If $\max\limits_{i}\mathbb E e^{-(\lambda+\varepsilon) t S_1^i}<\infty$ and $\mathbb E \max\limits_{i}e^{-(\lambda+\varepsilon) S_1^i}\mathbf 1_{\{N=1\}}<1$ for some $\varepsilon>0$, then $\max\limits_{i}\mathbb E [W_i(t)]^{-\lambda}<\infty$.
\end{itemize}
\end{co}

Corollary \ref{MCC2}(a) gives a sufficient condition for the existence of $\alpha$th-moment of $W_i(t)$. In fact, if we deal with the martingale $\{W_{n,i}(t)\}$ directly according to the ideas in the proof of Theorem \ref{MCT1}, the condition $p^{\alpha-1}\frac{\tilde{\rho}(\alpha t)}{\tilde \rho(t)^\alpha}<1$ can be weaken to $\frac{\tilde{\rho}(\alpha t)}{\tilde \rho(t)^\alpha}<1$ (see Huang \cite{huangT}, where we show that $\max\limits_{i}\mathbb E\left(W_{1,i}(t)\right)^\alpha<\infty$ and $\frac{\tilde{\rho}(\alpha t)}{\tilde \rho(t)^\alpha}<1$ is a necessary  and sufficient condition for $\max\limits_{i}\mathbb E [W_i(t)]^\alpha<\infty$.

\bigskip
The rest part of the paper is arranged as follows. In next section, we shall establish two auxiliary inequalities for the  martingale $\{\mathbf Y_n^{(t)}\}$, which will be used in Section 4 for the proof of Theorem \ref{MCT1}. In Section 5, we shall prove Theorems \ref{MCT2} and  \ref{MCT3}. Finally,  we shall consider the complex case in Section 6, where we  shall show sufficient conditions for the $L^\alpha$ convergence and  the $\alpha$th-moment of the Mandelbrot's martingale $\{\mathbf Y_n\}$.


\section {The martingale $\{\mathbf Y_n^{(t)}\}$}
The critical idea of the proof of Theorem \ref{MCT1} is to notice the double martingale structure (cf \cite{GK} for more information) of the martingale $\{\mathbf Y_n^{(t)}\}$ and apply the inequality  of martingale (Burkholder's inequality) to it. We shall go along the proof of Theorem \ref{MCT1} according to the lines of Huang \& Liu \cite{huang3} or Alsmeyer \emph{et al.} \cite{IK}.  In this section, we show two lemmas (inequalities) to the martingale $\{\mathbf Y_n^{(t)}\}$ which will be used in the proof of Theorem \ref{MCT1}.

\begin{lem}\label{MCL1.2.1}
Let $\alpha>1$. Fix $t\in \mathbb R$. If $\max\limits_i\mathbb E\left[Y_{1,i}^{(t)}\right]^\alpha<\infty$, then for each $i=1,\cdots,p$,
\begin{itemize}
\item[(a)] for $\alpha\in(1,2]$,
\begin{equation}\label{MCE1.2.1}
\mathbb E\left|Y_{n+1,i}^{(t)}-Y_{n,i}^{(t)}\right|^\alpha\leq C p^{(\alpha-1)n}\left[\frac{\rho(\alpha t)}{\rho(t)^\alpha}\right]^n;
\end{equation}
\item[(b)] for $\alpha>2$,
\begin{equation}\label{MCE1.2.2}
\mathbb E\left|Y_{n+1,i}^{(t)}-Y_{n,i}^{(t)}\right|^\alpha\leq C p^{\alpha n/2}\left[\frac{\rho(2 t)^{\alpha/2}}{\rho(t)^\alpha}\right]^n
\mathbb E[Y_{n,i}^{(2t)}]^{\alpha/2},
\end{equation}
\end{itemize}
where $C$ is a constant depending on $\alpha,p,t$.

\end{lem}

\begin{proof}
We can decompose $Y_{n,i}^{(t)}$  as
\begin{eqnarray*}
Y_{n,i}^{(t)}=\frac{1}{\rho(t)^n}\sum_{j=1}^{p}\sum_{u\in\T_n}(\textbf{X}_u^{(t)})_{ij}V_j(t)
=\frac{1}{\rho(t)^n}\sum_{j=1}^{p}\sum_{u\in\T_{n-1}}(\textbf{X}_u^{(t)})_{ij}Y_{1,j}^{(t)}(u),
\end{eqnarray*}
where $\mathbf Y_1^{(t)}(u)$ is a version of $\mathbf Y_1^{(t)}$ at root $u$. Hence
$$Y_{n+1,i}^{(t)}-Y_{n,i}^{(t)}=\frac{1}{\rho(t)^n}\sum_{j=1}^{p}
\sum_{u\in\T_n}(\textbf{X}_u^{(t)})_{ij}\left[Y_{1,j}^{(t)}(u)-V_j(t)\right].$$
By Burkholder's inequality (see for example \cite{chow}),
\begin{eqnarray*}
\mathbb E\left|Y_{n+1,i}^{(t)}-Y_{n,i}^{(t)}\right|^\alpha&\leq& \frac{p^{\alpha-1}}{\rho(t)^{\alpha n}}\sum_{j=1}^{p}\mathbb E\left|
\sum_{u\in\T_n}(\textbf{X}_u^{(t)})_{ij}\left[Y_{1,j}^{(t)}(u)-V_j(t)\right]\right|^\alpha\\
&\leq&\frac{C}{\rho(t)^{\alpha n}}\sum_{j=1}^{p}\mathbb E\left(\sum_{u\in\T_n}
\left[(\textbf{X}_u^{(t)})_{ij}\right]^2\left[Y_{1,j}^{(t)}(u)-V_j(t)\right]^2\right)^{\alpha/2}.
\end{eqnarray*}
Noticing the fact that
\begin{eqnarray*}
\mathbb E\left(\sum_{u\in\T_n}
\left[(\textbf{X}_u^{(t)})_{ij}\right]^2\left[Y_{1,j}^{(t)}(u)-V_j(t)\right]^2\right)^{\alpha/2}
\leq\left\{\begin{array}{ll}
\mathbb E\left(\sum\limits_{u\in \T_n}\left[(\textbf{X}_u^{(t)})_{ij}\right]^\alpha\right) \mathbb E\left|Y_{1,j}^{(t)}-V_j(t)\right|^\alpha, & \;\text{for $\alpha\in(1,2]$,}\\
\mathbb E\left(\sum\limits_{u\in \T_n}\left[(\textbf{X}_u^{(t)})_{ij}\right]^2\right)^{\alpha/2} \mathbb E\left|Y_{1,j}^{(t)}-V_j(t)\right|^\alpha,&\;\text{for $\alpha>2$},
\end{array}
\right.
\end{eqnarray*}
we have
\begin{eqnarray}\label{MCES31}
\mathbb E\left|Y_{n+1,i}^{(t)}-Y_{n,i}^{(t)}\right|^\alpha
\leq\left\{\begin{array}{ll}
\frac{C}{\rho(t)^{\alpha n}}\sum\limits_{j=1}^{p}\mathbb E\left(\sum\limits_{u\in \T_n}\left[(\textbf{X}_u^{(t)})_{ij}\right]^\alpha\right) \mathbb E\left|Y_{1,j}^{(t)}-V_j(t)\right|^\alpha, & \text{for $\alpha\in(1,2]$,}\\
\frac{C}{\rho(t)^{\alpha n}}\sum\limits_{j=1}^{p}\mathbb E\left(\sum\limits_{u\in \T_n}\left[(\textbf{X}_u^{(t)})_{ij}\right]^2\right)^{\alpha/2} \mathbb E\left|Y_{1,j}^{(t)}-V_j(t)\right|^\alpha,& \text{for $\alpha>2$}.
\end{array}
\right.
\end{eqnarray}
Note that for $u\in\T_n$,
\begin{equation}\label{MCE1.2.5}
\left[(\textbf{X}_u^{(t)})_{ij}\right]^\alpha\leq p^{(\alpha-1)(n-1)}(\textbf{X}_u^{(\alpha t)})_{ij},\qquad \forall \alpha>1,
\end{equation}
and
\begin{equation}\label{MCE1.2.6}
\sum_{u\in\T_n}(\textbf{X}_u^{(t)})_{ij}\leq \frac{\rho(t)^n}{V_j(t)}Y_{n,i}^{(t)}.
\end{equation}
Thus
\begin{eqnarray*}
\mathbb E\sum_{u\in\T_n}\left[(\textbf{X}_u^{(t)})_{ij}\right]^\alpha \leq p^{(\alpha-1)(n-1)}\mathbb E\sum_{u\in\T_n}(\textbf{X}_u^{(\alpha t)})_{ij}
\leq\frac{V_i(\alpha t)}{V_j(\alpha t)}p^{(\alpha-1)(n-1)}\rho(\alpha t)^n.
\end{eqnarray*}
Applying this inequality to the first inequality of (\ref{MCES31}), and noticing that
$\max_i\mathbb E\left[Y_{1,i}^{(t)}\right]^\alpha<\infty$, we obtain (\ref{MCE1.2.1}). To get (\ref{MCE1.2.2}), we only need
to see that
\begin{eqnarray*}
\mathbb E\left(\sum_{u\in\T_n}\left[(\textbf{X}_u^{(t)})_{ij}\right]^2\right)^{\alpha/2}&\leq&p^{\alpha (n-1)/2}\mathbb E\left(\sum_{u\in\T_n}(\textbf{X}_u^{(2 t)})_{ij}\right)^{\alpha/2}\\
&\leq&V_j(2 t)^{-\alpha/2} p^{\alpha (n-1)/2}\rho(2 t)^{\alpha n/2}\mathbb E\left[Y_{n,i}^{(2t)}\right]^{\alpha/2},
\end{eqnarray*}
and combing this inequality with the second inequality of (\ref{MCES31}).
\end{proof}

\begin{lem}\label{MCL1.2.2}
Let $\alpha>1$. Fix $t\in\mathbb R$. If $\max\limits_i\left[\E Y_{1,i}^{(t)}\right]^\alpha<\infty$, then for each $i=1,\cdots,p$,
\begin{equation}\label{MCE1.2.7}
\mathbb E\left[Y_{n,i}^{(t)}\right]^\alpha\leq Cn^{1+\frac{2^m-1}{2^m}\alpha}\left[\max\{1,\; p^{\alpha-1}\frac{\rho(\alpha t)}{\rho(t)^\alpha},\;
p^{\frac{2^l-1}{2^l}\alpha}\frac{\rho(2^lt)^{\alpha/2^l}}{\rho(t)^\alpha},l=1,2,\cdots,m\}\right]^n
\end{equation}
for $\alpha\in(2^m,2^{m+1}]$, where $m\geq0$ is an integer.
\end{lem}

\begin{proof}
At first, for $m=0$, $\alpha\in(1,2]$. Applying Burkholder's inequality to the martingale $\{Y_{n,i}^{(t)}\}$ and by
Lemma \ref{MCL1.2.1},
\begin{eqnarray*}
\mathbb E\left|Y_{n+1,i}^{(t)}-1\right|^\alpha&\leq& C\sum_{k=0}^{n-1}\mathbb E\left|Y_{k+1,i}^{(t)}-Y_{k,i}^{(t)}\right|^\alpha\\
&\leq&C\sum_{k=0}^{n-1}p^{(\alpha-1)k}\left(\frac{\rho(\alpha t)}{\rho(t)^\alpha}\right)^k\\
&\leq& Cn\left[\max\{1,p^{\alpha-1}\frac{\rho(\alpha t)}{\rho(t)^\alpha}\}\right]^n.
\end{eqnarray*}
Thus
$$\mathbb E\left[Y_{n,i}^{(t)}\right]^\alpha\leq Cn\left[\max\{1,p^{\alpha-1}\frac{\rho(\alpha t)}{\rho(t)^\alpha}\}\right]^n.$$
So (\ref{MCE1.2.7}) holds for $m=0$. Now suppose that (\ref{MCE1.2.7}) holds for some $m\geq0$, we shall
prove it still holds for $m+1$. For $\alpha\in(2^{m+1},2^{m+2}]$, we have $\alpha/2\in(2^{m},2^{m+1}]$.
Since $\max\limits_i\mathbb E\left[Y_{1,i}^{(t)}\right]^\alpha<\infty$ ensures that $\max\limits_i\mathbb E\left[Y_{1,i}^{(2t)}\right]^{\alpha/2}<\infty$, by induction, we have
\begin{equation}\label{MCE1.2.8}
\mathbb E\left[Y_{1,i}^{(2t)}\right]^{\alpha/2}\leq Ck^{1+\frac{2^m-1}{2^{m+1}}\alpha}\left[\max\{1,\; p^{\alpha/2-1}
\frac{\rho(\alpha t)}{\rho(2t)^{\alpha/2}},\;
p^{\frac{2^l-1}{2^{l+1}}\alpha}\frac{\rho(2^{l+1}t)^{\alpha/2^{l+1}}}{\rho(2t)^{\alpha/2}},l=1,2,\cdots,m\}\right]^k.
\end{equation}
Hence combing (\ref{MCE1.2.8}) with (\ref{MCE1.2.2}) we get
\begin{equation}\label{MCE1.2.9}
\mathbb E\left|Y_{k+1,i}^{(t)}-Y_{k,i}^{(t)}\right|^\alpha \leq Ck^{1+\frac{2^m-1}{2^{m+1}}\alpha}\left[\max\{ p^{\alpha-1}\frac{\rho(\alpha t)}{\rho(t)^\alpha},\;
p^{\frac{2^l-1}{2^l}\alpha}\frac{\rho(2^lt)^{\alpha/2^l}}{\rho(t)^\alpha},l=1,2,\cdots,m+1\}\right]^k.
\end{equation}
By Burkholder's inequality and Minkowski's inequality, and applying (\ref{MCE1.2.9}),

\begin{eqnarray*}
\mathbb E\left|Y_{n,i}^{(t)}-1\right|^\alpha&\leq&C\left(\sum_{k=0}^{n-1}\left(\mathbb E\left|Y_{k+1,i}^{(t)}-Y_{k,i}^{(t)}\right|^\alpha\right)^{2/\alpha}\right)^{\alpha/2}\\
&\leq& C\left(\sum_{k=0}^{n-1}k^{(1+\frac{2^m-1}{2^{m+1}}\alpha)\frac{2}{\alpha}}
\left[\max\{ p^{\alpha-1}\frac{\rho(\alpha t)}{\rho(t)^\alpha},\;p^{\frac{2^l-1}{2^l}\alpha}
\frac{\rho(2^lt)^{\alpha/2^l}}{\rho(t)^\alpha},l=1,\cdots,m+1\}\right]^{2k/\alpha}\right)^{\alpha/2}\\
&\leq& Cn^{1+\frac{2^m-1}{2^{m+1}}\alpha+\frac{\alpha}{2}}
\left[\max\{1, p^{\alpha-1}\frac{\rho(\alpha t)}{\rho(t)^\alpha},\;p^{\frac{2^l-1}{2^l}\alpha}
\frac{\rho(2^lt)^{\alpha/2^l}}{\rho(t)^\alpha},l=1,\cdots,m+1\}\right]^{n},
\end{eqnarray*}
which implies that (\ref{MCE1.2.7}) holds for $m+1$. This completes the proof.
\end{proof}

\bigskip

\noindent\textbf{Remark 3.1. } Lemmas \ref{MCL1.2.1}(b) and \ref{MCL1.2.2} also holds with $\beta$ in place of $2$ for any $\beta\in(1,2]$. To see this fact, observing that
$$\mathbb E\left(\sum_{u\in\T_n}\left[(\textbf{X}_u^{(t)})_{ij}\right]^2\right)^{\alpha/2}
\leq \mathbb E\left(\sum_{u\in\T_n}\left[(\textbf{X}_u^{(t)})_{ij}\right]^\beta\right)^{\alpha/\beta}$$
in the proof of Lemma \ref{MCL1.2.1}, one just need to repeat the proofs  of Lemmas \ref{MCL1.2.1}(b) and \ref{MCL1.2.2} with $\beta$ in place of $2$ for the case where $\alpha>2$.

\section {Proof of Theorem \ref{MCT1}}
Now we give the proof of Theorem \ref{MCT1}, by using the inequalities for  the martingale $\{\mathbf Y_n^{(t)}\}$ (Lemmas \ref{MCL1.2.1} and \ref{MCL1.2.2}) which are obtained in Section 3.

\begin{proof}[Proof of Theorem \ref{MCT1}]The proof of (a) is composed by two steps.

Step 1: we will show that if $\mathbb E\|\sum\limits_{k=1}^N\textbf{A}_k\|^\alpha<\infty$ and $p^{(\alpha-1)}\rho(\alpha)<1$,
then for each $i$, $\mathbb EY_i=V_i$ and  $\mathbb E[Y_i]^\alpha<\infty$, which implies that $\mathbb E\mathbf Y=\mathbf V$ and $0<\E\|\mathbf Y\|^\alpha<\infty$. In fact, it suffices to prove that $\sup\limits_n\mathbb E[Y_{n,i}]^\alpha<\infty$ for each $i$, which is equivalent to $Y_{n,i}\rightarrow Y_i$ in $L^\alpha$, so that $\mathbb EY_i=V_i$ and $0<\mathbb E[Y_i]^\alpha<\infty$. The condition $\E\|\sum\limits_{k=1}^N\textbf{A}_k\|^\alpha<\infty$, or equivalently, $\max\limits_i\mathbb E[Y_{1,i}]^\alpha<\infty$, ensures the finiteness of $\textbf{M}(t)$
for all $t\in[1,\alpha]$. Moreover, since $\textbf{M}(1)=\textbf{M}$ is strictly positive, by the log-convexity of $(\textbf{M}(t))_{ij}$, we have $\textbf{M}(t)$ is strictly positive for all $t\in[1,\alpha]$, so $\rho(t)$ exists for all $t\in[1,\alpha]$. For $\alpha\in(1,2]$, by Burkholder's inequality and Lemma \ref{MCL1.2.1},
$$\sup_n\mathbb E|Y_{n,i}-1|^\alpha\leq C\sum_{n=0}^{\infty}\mathbb E|Y_{n+1,i}-Y_{n,i}|^2\leq C\sum_{n=0}^{\infty}p^{(\alpha-1)n}\rho(\alpha)^n<\infty.$$
For $\alpha>2$, by Burkholder's inequality and Minkowski's inequality,
$$\sup_n\mathbb E|Y_{n,i}-1|^\alpha\leq C\left(\sum_{n=0}^{\infty}(\mathbb E|Y_{n+1,i}-Y_{n,i}|^\alpha)^{2/\alpha}\right)^{\alpha/2}.$$
We shall show the series $\sum\limits_{n=0}^{\infty}(\mathbb E|Y_{n+1,i}-Y_{n,i}|^\alpha)^{2/\alpha}<\infty$.
Observing that $\max\limits_i\mathbb E\left[Y_{1,i}^{(2)}\right]^{\alpha/2}<\infty$ since  $\max\limits_i\mathbb E(Y_{1,i})^{\alpha}<\infty$, by Lemma \ref{MCL1.2.2}, we have
for $\alpha\in(2^m,2^{m+1}]$ ($m\geq1$ is an integer),
\begin{equation}\label{MCE1.2.11}
\mathbb E\left[Y_{1,i}^{(2)}\right]^{\alpha/2}\leq Cn^\gamma\left[\max\{1,p^{\alpha/2-1}\frac{\rho(\alpha)}{\rho(2)^{\alpha/2}},
p^{\frac{2^l-1}{2^{l+1}}\alpha}\frac{\rho(2^{l+1})^{\alpha/2^{l+1}}}{\rho(2)^{\alpha/2}},l=1,\cdots,m-1\}\right]^n,
\end{equation}
where $\gamma=1+\frac{2^{m-1}-1}{2^m}\alpha\leq\frac{\alpha}{2}$.
By Lemma \ref{MCL1.2.1} and (\ref{MCE1.2.11}),
\begin{eqnarray*}
\mathbb E|Y_{n+1,i}-Y_{n,i}|^\alpha&\leq&Cp^{\alpha n/2}\rho(2)^{\alpha n/2}\mathbb E\left[(Y_n^{(2)})^i\right]^{\alpha/2}\\
&\leq&Cn^{\alpha/2}\left[\max\{p^{\alpha-1}\rho(\alpha),p^{\frac{2^l-1}{2^l}\alpha}\rho(2^l)^{\alpha/2^l},l=1,\cdots,m\}\right]^n.
\end{eqnarray*}
Therefore
$$\sum_{n=0}^{\infty}\left(\mathbb E|Y_{n+1,i}-Y_{n,i}|^\alpha\right)^{2/\alpha}\leq C\sum_nn\left[\max\{p^{\alpha-1}\rho(\alpha),p^{\frac{2^l-1}{2^l}\alpha}\rho(2^l)^{\alpha/2^l},l=1,\cdots,m\}\right]^{2n/\alpha}.$$
The series in the right side of the inequality above converges if and only if
\begin{equation}\label{MCE1.2.12}
\max\{p^{\alpha-1}\rho(\alpha),p^{\frac{2^l-1}{2^l}\alpha}\rho(2^l)^{\alpha/2^l},l=1,\cdots,m\}<1.
\end{equation}
Note that $\rho(t)$ is log-convex since $(\textbf{M}(t))_{ij}$ is log-convex (Kingman 1961).
We have $\forall \beta\in(1,\alpha)$, $\rho(\beta)\leq\rho(\alpha)^{(\beta-1)/(\alpha-1)}$. Thus
$$p^{\frac{\beta-1}{\beta}\alpha}\rho(\beta)^{\alpha/\beta}
\leq\left[p^{\alpha-1}\rho(\alpha)\right]^{\frac{\alpha(\beta-1)}{\beta(\alpha-1)}}<1,$$
and so (\ref{MCE1.2.12}) is true from this fact.

Step 2: we will prove that if $\E\|\sum\limits_{k=1}^N\textbf{A}_k\|^\alpha<\infty$ and $p^{(\alpha-1)}\rho_r(\alpha)<1$ for some $r$,
then for each $i$, $\mathbb EY_i=V_i$ and $\mathbb E[Y_i]^\alpha<\infty$. Let $\bar{N}:=N_r$ be the population of the $r$-generation and
$\bar{\textbf{A}}_i:=\textbf{X}_{u^i}$, where $u^i$ denotes the $i$-th particle of the $r$-generation. We consider
$(\bar{N},\bar{\textbf{A}}_1,\bar{\textbf{A}}_2,\cdots)$. Clearly, $\bar{\textbf{M}}:=\mathbb E\sum\limits_{i=1}^{\bar{N}}\bar{\textbf{A}}_i$ is finite and strictly positive with the maximum-modulus eigenvalue $1$ and the corresponding eigenvectors $\bar{\textbf{U}}=\textbf{U},\bar{\textbf{V}}=\textbf{V}$. Let $\mathbb{\bar T}$ be the corresponding Galton-Watson tree and $\mathbb{\bar T}_n=\{u\in \mathbb{\bar T}: |u|=n\}$. Define
\begin{equation*}
\bar{\textbf{Y}}_n:=\sum_{u\in\mathbb{\bar T}_n}\bar{\textbf{X}}_u\textbf{V}\quad\text{with}\;
\bar{\textbf{X}}_u:=\bar{\textbf{A}}_{u_1}\cdots\bar{\textbf{A}}_{u_1\cdots u_n}\;
\text{for}\;u\in\mathbb{\bar T}_n.
\end{equation*}
Similarly, we define $\bar{\textbf{M}}(t)$, $\bar{\textbf{Y}}_n(t)$, $\bar{\rho}(t)$ and $\bar{\textbf{V}}(t)$ like Section 2. It is easy to see that
$\bar{\textbf{Y}}_n$ has the same distribution as $\textbf{Y}_{nr}$, therefore,  $\bar{\textbf{Y}}:=\lim\limits_{n\rightarrow\infty}\bar{\textbf{Y}}_n$ a.s. has the same distribution as $\textbf{Y}$.  To get
$\mathbb EY_i=V_i$ and $\E[Y_i]^\alpha<\infty$, by Step 1, we only need to verify $\E\|\sum\limits_{i=1}^{\bar{N}}\bar{\textbf{A}}_i\|^\alpha<\infty$ and $p^{(\alpha-1)}\bar{\rho}(\alpha)<1$. The latter is obvious since $\bar{\textbf{M}}(t)=\textbf{M}_r(t)$ and so $\bar{\rho}(\alpha)=\rho_r(\alpha)$. To verify the former, we notice that $\mathbb E\|\sum\limits_{i=1}^{\bar{N}}\bar{\textbf{A}}_i\|^\alpha<\infty$ is equivalent to $\max\limits_i\mathbb E[\bar{Y}_{1,i}]^\alpha<\infty$,  which is true by Lemma \ref{MCL1.2.2}, since $\mathbb E[\bar {Y}_{1,i}]^\alpha=\mathbb E[ Y_{r,i}]^\alpha$.

Now we prove the converse (b). Suppose that $0<\E\|\mathbf Y\|^\alpha<\infty$, which implies that $\max\limits_i\mathbb E[Y_i]^\alpha<\infty$ and $\mathbf Y$ is non-degenerate. As $\mathbf Y$  is a non-trivial solution of the equation (\ref{E}), we have
$\E \mathbf Y=\mathbf M \E\mathbf Y$ with $\E \mathbf Y\neq \mathbf 0$, which means that $\E \mathbf Y$ is a  non-trivial eigenvector corresponding to the eigenvalue $1$, and so $\E \mathbf Y=c\mathbf Y$ for some constant $c>0$. By equation (\ref{E}), for each $i$,
$$Y_i=\sum_{k=1}^N(\mathbf A_k \mathbf Y(k))_i=\sum_{k=1}^N\sum_{j=1}^p(\mathbf A_k)_{ij}Y_i(k).$$
By Jensen's inequality, for each $i$,
\begin{eqnarray*}
\E[Y_i]^\alpha&\geq & \E\left[\E\left(\sum_{k=1}^N\sum_{j=1}^p(\mathbf A_k)_{ij}Y_j(k)\big{|}\mathcal F_1\right)\right]^\alpha\\
&=&\E\left[\E\sum_{k=1}^N\sum_{j=1}^p(\mathbf A_k)_{ij}\E Y_j\right]^\alpha
\\&=&c^\alpha\E\left[\E\sum_{k=1}^N(\mathbf A_k \mathbf V)_i\right]^\alpha\\
&=&c^\alpha\E[Y_{1,i}]^\alpha.
\end{eqnarray*}
Thus $\max\limits_i\mathbb E[Y_i]^\alpha<\infty$ implies $\max\limits_i\mathbb E[Y_{1,i}]^\alpha<\infty$, or equivalently, $\E\|\sum_{i=1}^N\textbf{A}_i\|^\alpha<\infty$.

Next, we consider $\rho_n(t)$.
Since
\begin{eqnarray}\label{MCE1.2.13}
[Y_i]^\alpha=\left[\sum_{j=1}^p\sum_{u\in\T_n}(\textbf{X}_u)_{ij}Y_j(u)\right]^\alpha\geq
\sum_{j=1}^p\sum_{u\in\T_n}\left[(\textbf{X}_u)_{ij}\right]^\alpha\left[ Y_j(u)\right]^\alpha,
\end{eqnarray}
we obtain
\begin{equation} \label{MCE1.2.14}
\E[Y_i]^\alpha \geq
\sum_{j=1}^p\E\sum_{u\in\T_n}\left[(\textbf{X}_u)_{ij}\right]^\alpha \E[ Y_j]^\alpha=\sum_{j=1}^p(\textbf{M}_n(\alpha))_{ij} \E[ Y_j]^\alpha.
\end{equation}
Thus
\begin{equation} \label{MCE1.2.15}
\sum_{i=1}^pU_{n,i}(\alpha)\E[Y_i]^\alpha\geq\sum_{j=1}^p \E[ Y_j]^\alpha\sum_{i=1}^pU_{n,i}(\alpha)
(\textbf{M}_n(\alpha))_{ij}=\rho_n(\alpha)\sum_{j=1}^pU_{n,j}(\alpha)\E[Y_j]^\alpha,
\end{equation}
which leads to $\rho_n(\alpha)\leq1$. If additionally (\ref{MCEC2}) holds, then for each $i$, $\P\left(\sum\limits_{j=1}^p\sum\limits_{u\in \T_n} \mathbf{1}_{\{(\textbf{X}_u)_{ij}>0\}}=0\; or \;1\right)<1$. Hence the
strictly inequality in (\ref{MCE1.2.13})
holds with positive probability, and so both (\ref{MCE1.2.14}) and (\ref{MCE1.2.15}) are  strictly inequalities, which leads to $\rho_n(\alpha)<1$.

\end{proof}


\section {Proof of Theorems \ref{MCT2} and \ref{MCT3}}

We will prove  Theorems \ref{MCT2} and \ref{MCT3}  based on the equation (\ref{E}), with ideas  from Liu \cite{liu2}. Recall that $\phi(\mathbf t)=\E e^{-\mathbf t\cdot\mathbf Z}$ is the Laplace transform of the non-trivial solution $\mathbf Z$ to the equation (\ref{E}). By (\ref{E}), $\phi(\mathbf t)$ satisfies  the functional equation
\begin{equation}\label{E1}
\phi(\mathbf t)=\E\prod_{k=1}^N\phi(\mathbf t \mathbf A_k).
\end{equation}
Our proofs are based on this equation.

\bigskip
To prove Theorem \ref{MCT2}, the two lemmas below  are  necessary.

\begin{lem}\label{MCLN1}
Let $\phi:\mathbb R_+^p\mapsto\mathbb R_+$ be a bounded function, and $\mathbf A=(a_{ij})$ be a non-zero matrix such that for some $0<q<1$, $t_\varepsilon>0$ and all $\mathbf t$ satisfying $\|\mathbf t\|>t_\varepsilon$,
\begin{equation}\label{MCLE41}
\phi(\mathbf t)\leq q \E\phi(\mathbf t\mathbf A).
\end{equation}
 If $q\E\left(\min\limits_i\sum\limits_j a_{ij}\right)^{-\lambda}<1$, then
 $\phi(\mathbf t)=O(\|\mathbf t\|^{-\lambda})$ ($\|\mathbf t\|\rightarrow\infty$).
\end{lem}

\begin{proof}
Assume that $\phi$ is bounded by a constant $K$. For  $\mathbf t\neq \mathbf 0$, if $\|\mathbf t\|\leq t_\varepsilon$, then $\phi(\mathbf t)\leq K\leq K t_\varepsilon^\lambda \|\mathbf t\|^{-\lambda}$, which yields by (\ref{MCLE41})
\begin{equation}\label{MCLE42}
\phi(\mathbf t)\leq q \E\phi(\mathbf t\mathbf A)+ C  \|\mathbf t\|^{-\lambda},\quad \text{for all $\mathbf t\neq \mathbf 0$,}
\end{equation}
where $C$ is a general positive constant. Let $\{\mathbf A_k\}$ be a family of i.i.d copies of $\mathbf A$. By induction on (\ref{MCLE42}),
\begin{equation}\label{MCLE43}
\phi(\mathbf t)\leq q^n \E\phi(\mathbf t\mathbf A_1\cdots \mathbf A_n)+ C  \left[\sum_{k=1}^{n-1}\left(q^{k-1}\E\|\mathbf t\mathbf A_1\cdots \mathbf A_{k-1}\|\right)^{-\lambda}+\|\mathbf t\|^{-\lambda}\right],\quad \text{for all $\mathbf t\neq \mathbf 0$.}
\end{equation}
Note that for any matrix $\mathbf A=(a_{ij})$ and vector $\mathbf t$, we have
$$\|\mathbf t\mathbf A\|=\sum_j(\mathbf t \mathbf A)_j=\sum_j\sum_i t_i a_{ij}\geq \sum_i t_i \min_i\sum_j a_{ij}=\|\mathbf t\|\left(\min_i\sum_j a_{ij}\right).$$
Thus by the independency of $\{\mathbf A_k\}$,
\begin{equation}\label{MCLE44}
\E\|\mathbf t\mathbf A_1\cdots \mathbf A_{k}\|^{-\lambda}
\leq \|\mathbf t\|^{-\lambda}\E\left(\prod_{l=1}^k\left(\min_i\sum_j(\mathbf A_l)_{ij}\right)\right)^{-\lambda}=\|\mathbf t\|^{-\lambda}\left[\E\left(\min_i\sum_j a_{ij}\right)^{-\lambda }\right]^k.
\end{equation}
Combing (\ref{MCLE44}) with (\ref{MCLE43}) and letting $n\rightarrow\infty$ leads to $\phi(\mathbf t)=O(\|\mathbf t\|^{-\lambda})$ ($\|\mathbf t\|\rightarrow\infty$).
\end{proof}

\bigskip

\begin{lem}[\cite{liu1999}, Lemma 4.4]\label{MCLN2}
\label{MDL3.2} Let $X$ be a positive
random variable. For $0<a<\infty$, consider the following
statements:\newline
\begin{equation*}
\begin{array}{ll}
(i)\;\mathbb{E}X^{-a}<\infty; & (ii)\;\mathbb{E}e^{-tX}=O(t^{-a})(t\rightarrow\infty); \\
(iii)\;\mathbb{P}(X\leq x)=O(x^a)(x\rightarrow0); & (iv)\;\forall b\in(0,a),
\mathbb{E}X^{-b}<\infty. \\
\end{array}
\end{equation*}
Then the following implications hold: (i) $\Rightarrow$ (ii) $\Leftrightarrow$ (iii) $\Rightarrow$ (iv).
\end{lem}

 \medskip
 \begin{proof}[Proof of  Theorem \ref{MCT2}]
 Let $N_{\delta}=\sum\limits_{k=1}^N\mathbf 1_{\{\min\limits_{i}\sum\limits_{j} ( \mathbf A_k)_{ij}>\delta\}}$ for $\delta>0$. Then $N_\delta \uparrow N$, as $\delta\downarrow0$. Since $\phi(\mathbf t)=\mathbb E e^{-\mathbf t\cdot \mathbf Z}\rightarrow 0$ as $\|\mathbf t\|\rightarrow\infty$, there exists $t_\varepsilon>0$ such that for $\|\mathbf t\|>t_\varepsilon$, $\phi(\mathbf t)<\varepsilon$.
 For $\|\mathbf t\|>t_\varepsilon/\delta$, if $\min\limits_{i}\sum\limits_{j} ( \mathbf A_k)_{ij}>\delta$, we have
 $\|\mathbf t \mathbf A_k\|\geq \|\mathbf t\|\min\limits_{i}\sum\limits_{j} ( \mathbf A_k)_{ij}>t_\varepsilon$. By equation (\ref{E1}),
 \begin{eqnarray*}
 \phi(\mathbf t)\leq \E\phi(\mathbf t\mathbf A_1)\left(\varepsilon^{N_\delta-1}\mathbf 1_{\{N_\delta\geq1\}}+\mathbf 1_{\{N_\delta=0\}}\right)=q_{\varepsilon,\delta}\mathbb E \phi(\mathbf t\tilde{\mathbf A}),
 \end{eqnarray*}
 where $q_{\varepsilon,\delta}=\E\left(\varepsilon^{N_\delta-1}\mathbf 1_{\{N_\delta\geq1\}}+\mathbf 1_{\{N_\delta=0\}}\right)$ and $\tilde{\mathbf A}=(\tilde a_{ij})$ is a random matrix whose distribution is determined by $\mathbb E g(\tilde{\mathbf A})=\frac{1}{q_{\varepsilon, \delta}}\E g({\mathbf A}_1)\left(\varepsilon^{N_\delta-1}\mathbf 1_{\{N_\delta\geq1\}}+\mathbf 1_{\{N_\delta=0\}}\right)$ for all bounded and measurable function $g$ on $\mathbb R_+^{p^2}$. We can see that by the dominated convergence theorem,
 $$q_{\varepsilon,\delta}\stackrel{\delta\downarrow0}{\longrightarrow}\E\varepsilon^{N-1}\mathbf 1_{\{N\geq1\}}\stackrel{\varepsilon\downarrow0}{\longrightarrow}\P(N=1)<1,$$
 and  since $\E\left(\min\limits_{i}\sum\limits_{j} a_{ij}\right)^{-\lambda}<\infty$,
\begin{eqnarray*}
 q_{\varepsilon,\delta}\E\left(\min\limits_{i}\sum\limits_{j}\tilde{ a}_{ij}\right)^{-\lambda}&=&\E\left(\min\limits_{i}\sum\limits_{j} a_{ij}\right)^{-\lambda}\left(\varepsilon^{N_\delta-1}\mathbf 1_{\{N_\delta\geq1\}}+\mathbf 1_{\{N_\delta=0\}}\right)\\
 &\stackrel{\delta\downarrow0}{\longrightarrow}&\E\left(\min\limits_{i}\sum\limits_{j} a_{ij}\right)^{-\lambda}\mathbf 1_{\{N\geq1\}}\stackrel{\varepsilon\downarrow0}{\longrightarrow}\E\left(\min\limits_{i}\sum\limits_{j} a_{ij}\right)^{-\lambda}\mathbf 1_{\{N=1\}}<1.
 \end{eqnarray*}
 By Lemma \ref{MCLN1}, $\phi(\mathbf t)=O(\|\mathbf t\|^{-\lambda})$ ($\|\mathbf t\|\rightarrow\infty$). Thus for given non-zero $\mathbf y=(y^1,\cdots,y^p)\in \mathbb R_{+}^p$, $\E e^{-t\mathbf y\cdot \mathbf Z}=O(t^{-\lambda})(t\rightarrow\infty)$, so that by Lemma \ref{MCLN2},
 $\P(\mathbf y\cdot \mathbf Z\leq x)=O(x^\lambda) (x\rightarrow0)$ and $\E(\mathbf y\cdot \mathbf Z)^{-\lambda_1},\infty$, $\forall 0<\lambda_1<\lambda$.

 For the second part, notice  that we have obtained $ \phi(\mathbf t)\leq C \|\mathbf t\|^\lambda$ for all $\|\mathbf t\|>0$ in the first part, where $C$ is a positive constant. By equation (\ref{E1}),
 $$\phi(\mathbf t)\leq \E\prod_{k=1}^{\underline{m}}\phi(\mathbf t\mathbf A_k)\leq C^{\underline{m}}\E\prod_{k=1}^{\underline{m}}\|\mathbf t\mathbf A_k\|^{-\lambda}\leq C^{\underline{m}}\|\mathbf t\|^{-\underline{m}\lambda} \mathbb E\left[\prod\limits_{k=1}^{\underline {m}}\left(\min\limits_i\sum\limits_{j=1}^p (\mathbf A_k)_{ij}\right)^{-\lambda}\right].$$
 The rest results follow by Lemma \ref{MCLN2}.

\end{proof}

\bigskip
 \begin{proof}[Proof of  Theorem \ref{MCT3}]We only prove the results for $\phi(\mathbf t)$. The assertions for $\P(\mathbf y\cdot \mathbf Z\leq x)$ follow from  that about $\E^{-t\mathbf y\cdot \mathbf Z}$ and the Tauberian Theorem of exponential type (cf  \cite{liu96}).

 We first prove (a). By equation (\ref{E}),
$$1=\sum_{j=1}^p V_j=\E\sum_{k=1}^N\sum_{j,l=1}^p (\mathbf A_k)_{jl}V_l>\E\sum_{k=1}^{\underline{m}}\sum_{j,l=1}^p (\mathbf A_k)_{jl}V_l\geq \underline{a}\;\underline{m}\;p.$$
The strict inequality holds because of (\ref{MCEC2}) and $\P(N=\underline{m})>0$. Therefore, we have $\gamma=-\log\underline{m}/\log{(\underline{a}p)}\in(0,1)$. Since for all $k$,
$$\phi(\mathbf t\mathbf A_k)=\E\exp\left\{-\sum_{i,j=1}^pt_i(\mathbf A_k)_{ij}Z_j\right\}\leq \E\exp\{-\underline{a}\|\mathbf t\|\mathbf e\cdot\mathbf Z\}=\phi(\underline{a}\|\mathbf t\|\mathbf e),$$
where $\mathbf e=(1,1,\cdots, 1)$,
by equation (\ref{E1}),
\begin{equation}\label{MCNT21}
\phi(\mathbf t)\leq \E\prod_{k=1}^{\underline{m}}\phi(\mathbf t \mathbf A_k)\leq \left[\phi(\underline{a}\|\mathbf t\|\mathbf e)\right]^{\underline{m}}.
\end{equation}
Applying (\ref{MCNT21}) with $\mathbf t=\underline{a}\|\mathbf t\|\mathbf e$, we have $\phi(\underline{a}\|\mathbf t\|\mathbf e)\leq \left[\phi(\underline{a}^2p\|\mathbf t\|\mathbf e)\right]^{\underline{m}}$, so that
$\phi(\mathbf t)\leq \left[\phi(\underline{a}^2p\|\mathbf t\|\mathbf e)\right]^{\underline{m}}$. By iteration, we get
\begin{equation}\label{MCNT22}
\phi(\mathbf t)\leq \left[\phi(\underline{a}^kp^{k-1}\|\mathbf t\|\mathbf e)\right]^{\underline{m}^k}.
\end{equation}
As $\underline{a}p<1$, for $\|\mathbf t\|\geq p$, there exists an integer $k\geq0$ such that  $p/(\underline{a}p)^k\leq\|\mathbf t\|< p/(\underline{a}p)^{k+1}$. So this $k$ satisfies
\begin{equation}\label{MCNT24}
\frac{\log p-\log \|\mathbf t\|}{\log {(\underline{a}p})}-1<k\leq \frac{\log p-\log \|\mathbf t\|}{\log {(\underline{a}p})}.
\end{equation}
For any $x\geq1$, one can see that
\begin{equation}\label{MCNT23}
\phi(x\mathbf e)=\E\exp\left\{-x\sum_{j=1}^pZ_j\right\}\leq \E\exp\{-\sum_{j=1}^pZ_j\}=\phi(\mathbf e)<1.
\end{equation}
Since $\underline{a}^kp^{k-1}\|\mathbf t\|\geq 1$, by (\ref{MCNT22}), (\ref{MCNT23}) and (\ref{MCNT24}) , we have
$$\log \phi(\mathbf t)\leq \underline{m}^k\log \phi(\mathbf e)\leq \exp\left\{\frac{\log\underline{m}}{\log(\underline{a}p)}(\log p-\log\|\mathbf t\|)\right\}=-C_1\|\mathbf t\|^\gamma,$$
where $C_1=- p^{-\gamma}\log \phi(\mathbf e)>0 $.

We then prove (b). The proof is similar to that of (a). If $\max_{ij}(\mathbf A_k)_{ij}\leq \underline{a}+\varepsilon$, then $$\phi(\mathbf t\mathbf A_k)\geq\phi\left((\underline{a}+\varepsilon)\|\mathbf t\|\mathbf e\right).$$
By equation (\ref{E1}),
\begin{equation}\label{MCNT25}
\phi(\mathbf t)\geq \E\prod_{k=1}^{\underline{m}}\phi(\mathbf t \mathbf A_k)\mathbf 1_{\{N=\underline{m},\; \max\limits_{ij}(\mathbf A_k)_{ij}\leq \underline{a}+\varepsilon,\forall k\}}
\geq \rho\left[\phi((\underline{a}+\varepsilon)\|\mathbf t\|\mathbf e)\right]^{\underline{m}},
\end{equation}
where $\rho=\P(N=\underline{m},\; \max\limits_{ij}(\mathbf A_k)_{ij}\leq \underline{a}+\varepsilon,\forall k)<1$. By iteration, we get
\begin{equation}\label{MCNT26}
\phi(\mathbf t)\geq \rho^{\sum_{j=0}^{k-1}\underline{m}^j}\left[\phi((\underline{a}+\varepsilon)^kp^{k-1}\|\mathbf t\|\mathbf e)\right]^{\underline{m}^k}.
\end{equation}
As $(\underline{a}+\varepsilon)p<1$, for $\|\mathbf t\|\geq p$, there exists an integer $k\geq1$ such that  $p/((\underline{a}+\varepsilon)p)^{k-1}\leq\|\mathbf t\|< p/((\underline{a}+\varepsilon)p)^{k}$. Since $\phi(x\mathbf e)\geq \phi(\mathbf e)$ for any $x<1$ and $(\underline{a}+\varepsilon)^kp^{k-1}\|\mathbf t\|< 1$, we have
$$\phi((\underline{a}+\varepsilon)^kp^{k-1}\|\mathbf t\|\mathbf e)\geq \phi(\mathbf e).$$
Therefore, (\ref{MCNT26}) yields
\begin{eqnarray*}
\log \phi(\mathbf t)&\geq &\underline{m}^k\left(\log \phi(\mathbf e)+ \underline{m}^{-k} \sum\limits_{j=0}^{k-1}\underline{m}^j\log \rho \right)\\
&\geq &\underline{m}^k\left(\log \phi(\mathbf e)+ \frac{\log \rho}{\underline{m}-1} \right)\\
&\geq& \exp\left\{\frac{\log\underline{m}}{\log((\underline{a}+\varepsilon)p)}(\log p-\log\|\mathbf t\|)+\log \underline{m}\right\}\left(\log \phi(\mathbf e)+ \frac{\log \rho}{\underline{m}-1} \right)\\
&=&-C_2\|\mathbf t\|^{\gamma(\varepsilon)},
\end{eqnarray*}
where $C_2=- p^{-\gamma(\varepsilon)}\underline{m}\left(\log \phi(\mathbf e)+ \frac{\log \rho}{\underline{m}-1} \right)>0$.
\end{proof}


\bigskip

\section {Moments for the complex case}
In this section, we consider  the complex case, where in equation (\ref{E}),
all the matrix $\mathbf A_k$ and the vectors $\mathbf Z, \mathbf Z(k)$ are complex (with $\mathbb C$ in place of $\mathbb R_+$). Here we still interested in the existence of the $\alpha$th-moment ($\alpha>1$) solution, or in other words, the $L^\alpha$ convergence and the $\alpha$th-moment of the Mandelbrot's martingale $\{\mathbf Y_n\}$ defined by (\ref{beqYn}).

Besides Assumption (H), we assume moreover that
$$\hat{\textbf{M}}:=\E\sum\limits_{i=1}^{N}\hat{\textbf{A}}_k\quad\text{with}\;(\hat{\textbf{A}}_k)_{ij}:=|(\textbf{A}_k)_{ij}|$$
is finite and strictly positive. For $t\in\mathbb{R}$ fixed,   let
$$\hat{\textbf{M}}(t):=\E\sum_{i=1}^N\hat{\textbf{A}}_k^{(t)}\quad\text{with}\; (\hat{\textbf{A}}_k^{(t)})_{ij}:=(\hat{\textbf{A}}_k)^t_{ij},$$
whose  maximum-modulus eigenvalue is denoted by $\hat\rho(t)$ and the
corresponding normalized left and right  positive eigenvectors by $\hat{\textbf{U}}(t),\hat{\textbf{V}}(t)$. Define
$$\hat{\textbf{Y}}_n^{(t)}:=\frac{\sum\limits_{u\in\T_n}\hat{\textbf{X}}_u^{(t)}\hat{\textbf{V}}(t)}{\hat\rho(t)^n}\quad\text{with}\;
\hat{\textbf{X}}_u^{(t)}:=\hat{\textbf{A}}_{u_1}^{(t)}\cdots\hat{\textbf{A}}_{u_1\cdots u_n}^{(t)}\;
\text{for}\;u\in\T_n.$$
Obviously, $\{\hat{\textbf{Y}}_n^{(t)}\}$ has the same structure as the martingale $\{\textbf{Y}_n^{(t)}\}$ of the real case for which  we have established inequalities  in Section 3,  therefore, we
can apply these results (Lemmas \ref{MCL1.2.1} and \ref{MCL1.2.2}) to the martingale $\{\hat{\textbf{Y}}_n^{(t)}\}$.

\bigskip

Following similar arguments to the proof of Theorem \ref{MCT1}, we  reach  the following result for the
the complex case.

\begin{thm}[Complex case]\label{MCT21}Assume that all the matrix $\mathbf A_k$ and the vectors $\mathbf Z, \mathbf Z(k)$ are complex.
Let $\alpha>1$. If $\E\|\sum\limits_{i=1}^N\hat{\textbf{A}_i} \|^\alpha<\infty$ and either of the following assertions holds:
\begin{itemize}
\item[(i)] $\alpha \in(1,2]$ and $p^{(\alpha-1)}\hat{\rho}(\alpha)<1$;
\item[(ii)] $\alpha>2$ and $\max\{p^{\alpha-1}\hat{\rho}(\alpha), p^{\alpha/\beta}\hat{\rho}(\beta)\}<1$ for some $\beta\in(1,2]$,
\end{itemize}
then
$\sup\limits_n\E\|\mathbf Y_n\|^\alpha<\infty$, and $\{\mathbf Y_n\}$ converges a.s. and in $L^\alpha$ to a random vector $\mathbf Y$, so that
$\E\mathbf Y=\mathbf V$ and $0<\E\|\mathbf Y\|^\alpha<\infty$.
\end{thm}

In particular, for the case $p=1$, it is easy to see that $V=1$ and $\hat{\rho}(t)=\hat{m}(t)=\E\sum\limits_{i=1}^N|A_i|^t$.

\begin{co}[case p=1]Let $p=1$ and $\alpha>1$. If $\E \left(\sum\limits_{i=1}^N|A_i|\right)^\alpha<\infty$ and either of the following assertions holds:\\
\begin{itemize}
\item[(i)] $\alpha \in(1,2]$ and $ \hat{\rho}(\alpha)<1$;
\item[(ii)]$\alpha>2$ and $\max\{ \hat{\rho}(\alpha), \hat{\rho}(\beta)\}<1$ for some $\beta\in(1,2]$,
\end{itemize}
then
$\sup_n\E|Y_n|^\alpha<\infty$ and $\{Y_n\}$ converges a.s. and in $L^\alpha$ to a random variable $Y$, so that
$\E Y=1$ and $0<\E Y^\alpha<\infty$.
\end{co}

\bigskip
\bigskip
The proof of Theorem \ref{MCT21} is similar to that  of Theorem \ref{MCT1}. We first show several lemmas for the martingale $\{\mathbf Y_n^{(t)}\}$.

\begin{lem}\label{MCL2.1}Let $\alpha>1$. Fix $t\in\mathbb R$. Assume that $\max\limits_i\E|Y_{1,i}^{(t)}|^\alpha<\infty$. Then for each $i=1,2,\cdots, p$,
\begin{itemize}
\item[(a)]if $\alpha\in(1,2]$,
\begin{equation}\label{MCE2.2.1}
\E\left|Y_{n+1,i}^{(t)}-Y_{n,i}^{(t)}\right|^\alpha\leq C p^{(\alpha-1)n}\left[\frac{\hat{\rho}(\alpha t)}{|\rho(t)|^\alpha}\right]^n;
\end{equation}
\item[(b)]if $\alpha>2$,   for any $\beta\in(1,2]$,
\begin{equation}\label{MCE2.2.2}
\E\left|Y_{n+1,i}^{(t)}-Y_{n,i}^{(t)}\right|^\alpha\leq C p^{\alpha n/2}\left[\frac{\hat{\rho}(\beta t)^{\alpha/\beta}}{|\rho(t)|^\alpha}\right]^n
\E\left[\hat{Y}_{n,i}^{(\beta t)}\right]^{\alpha/\beta },
\end{equation}
\end{itemize}
where $C$ is a constant depending on $\alpha,p,t$.

\end{lem}

\begin{proof}
Notice that $|(\textbf{X}_u^{(t)})_{ij}|\leq(\hat{\textbf{X}}_u^{(t)})_{ij}$. Applying Burkholder's inequality, we get
\begin{eqnarray*}
\E\left|Y_{n+1,i}^{(t)}-Y_{n,i}^{(t)}\right|^\alpha&\leq& \frac{C}{|\rho(t)|^{\alpha n}}\sum_{j=1}^p\E\left(
\sum_{u\in\T_n}\left|(\textbf{X}_u^{(t)})_{ij}\right|^2\left|Y_{1,j}^{(t)}(u)-V_j(t)\right|^2\right)^{\alpha/2}\\
&\leq&\frac{C}{|\rho(t)|^{\alpha n}}\sum_{j=1}^p\E\left(
\sum_{u\in\T_n}\left[(\hat{\textbf{X}}_u^{(t)})_{ij}\right]^2\left|Y_{1,j}^{(t)}(u)-V_j(t)\right|^2\right)^{\alpha/2}.
\end{eqnarray*}
Then repeat the proof of Lemma \ref{MCL1.2.1} (with $\beta$ in place of $2$ for the case where $\alpha>2$).
\end{proof}

Apply Lemma \ref{MCL1.2.2} (with $\beta$ in place of $2$) to (\ref{MCE2.2.2}), we immediately get the following lemma.

\begin{lem}\label{MCL2.2}
Let $\alpha>1$. Fix $t\in\mathbb R$. If $\max\limits_i\E\left|Y_{1,i}^{(t)}\right|^\alpha<\infty$ and $\max\limits_i\E\left|\hat{Y}_{1,i}^{(\beta t)}\right|^{\alpha/\beta}<\infty$ for some $\beta\in(1,2]$,
then for $\alpha\in(\beta^m,\beta^{m+1}]$ ($m\geq 1$ is an integer),
$$\E\left|Y_{n+1,i}^{(t)}-Y_{n,i}^{(t)}\right|^\alpha\leq C n^{\alpha/\beta}\left[\max\{p^{\alpha-1}\frac{\hat{\rho}(\alpha t)}{|\rho(t)|^\alpha},\;
p^{\frac{\beta^l-1}{\beta^l}\alpha}\frac{\hat{\rho}(\beta^lt)^{\alpha/\beta^l}}{|\rho(t)|^\alpha},l=1,\cdots,m
\}\right]^n.$$
\end{lem}

\bigskip

Combing Lemmas \ref{MCL2.1} and \ref{MCL2.2}  leads to Lemma \ref{MCL2.3} below.

\begin{lem}\label{MCL2.3}
Let $\alpha>1$. Assume that $\E\|\sum\limits_{k=1}^N\hat{\textbf{A}_k}\|^\alpha<\infty$. Then for each $i=1,2,\cdots, p$,
\begin{itemize}
\item[(a)]if $\alpha\in(1,2]$,
\begin{equation}\label{MCE2.2.3}
\E\left|Y_{n+1,i} -Y_{n,i} \right|^\alpha\leq C \left[ p^{(\alpha-1)}\hat{\rho}(\alpha ) \right]^n;
\end{equation}
\item[(b)]if $\alpha>2$,  for any $\beta\in(1,2]$,
\begin{equation}\label{MCE2.2.4}
\E\left|Y_{n+1,i} -Y_{n,i} \right|^\alpha\leq C n^{\alpha/\beta}\max\{p^{\alpha-1}\hat{\rho}(\alpha), p^{\alpha/\beta}\hat{\rho}(\beta)^{\alpha/\beta}\}^n,
\end{equation}
\end{itemize}
where $C$ is a constant depending on $\alpha,p,t$.
\end{lem}

\begin{proof}Firstly, we remark that $\hat{\rho}(t)$ exists for all $t\in[1,\alpha]$ since
$\E\|\sum\limits_{k=1}^N\hat{\textbf{A}}_k\|^\alpha<\infty$ and $\hat{\textbf{M}}(1)=\hat{\textbf{M}}$ is finite and strictly positive.
Furthermore, $\E\|\sum\limits_{k=1}^N\hat{\textbf{A}}_k\|^\alpha<\infty$ implies that $\max\limits_i\E|Y_{1,i}|^\alpha<\infty$ and $\max\limits_i\E|\hat{Y}_{1,i}^{(\beta)}|^{\alpha/\beta}<\infty$. So (\ref{MCE2.2.3}) is directly from (\ref{MCE2.2.1}). For $\alpha>2$,
 by Lemma \ref{MCL2.2}, we have
\begin{eqnarray*}
\E\left|Y_{n+1,i} -Y_{n,i} \right|^\alpha&\leq &C n^{\alpha/\beta}\left[\max\{p^{\alpha-1}\hat{\rho}(\alpha ) ,\;
p^{\frac{\beta^l-1}{\beta^l}\alpha}\hat{\rho}(\beta^l)^{\alpha/\beta^l} ,l=1,\cdots,m\}\right]^n\\
&\leq&C n^{\alpha/\beta}\left(\sup_{\beta\leq x\leq\alpha}\{p^{1-1/x}\hat{\rho}(x)^{1/x}\}\right)^{\alpha n},
\end{eqnarray*}
if $\alpha\in(\beta^m,\beta^{m+1}]$ ($m\geq 1$ is an integer). Let
$$g(x):=\log(p^{1-1/x}\hat{\rho}(x)^{1/x})=(1-\frac{1}{x})\log p+\frac{1}{x}\log\hat{\rho}(x).$$
Clearly, $g(x)$ is derivable on
$(1,\alpha)$ with derivative
$$g'(x)=\frac{h(x)}{x^2},\quad \text{where} \; h(x):=\log p+x\frac{\hat{\rho}'(x)}{\hat{\rho}(x)}-\log\hat{ \rho}(x).$$
The log-convexity of $\hat{\rho}(x)$ implies that
$h(x)$ is increasing, hence $g(x)$ reaches its maximum on a closed
interval at the extremity points. We have
$$ \sup_{\beta\leq x\leq\alpha}\{p^{1-1/x}\hat{\rho}(x)^{1/x}\}=\max\{p^{1-1/\alpha}\hat{\rho}(\alpha)^{1/\alpha},
p^{1/\beta}\hat{\rho}(\beta)^{1/\beta}\}. $$
The proof is complete.
\end{proof}

\bigskip

Now we prove Theorem \ref{MCT21}.

\begin{proof}[Proof of Theorem \ref{MCT21}]
By Lemma \ref{MCL2.3}, we can obtain the series $\sum_{n}\left(\E|Y_{n+1,i}-Y_{n,i}|^\alpha\right)^{1/\alpha}<\infty$. Observing
that
$$\left(\E|Y_{n,i}|^\alpha\right)^{1/\alpha}\leq\sum_{k=0}^{n-1}\left(\E|Y_{k+1,i}-Y_{k,i}|^\alpha \right)^{1/\alpha}+1,$$
we immediately get
$$\sup_n\E|Y_{n,i}|^\alpha\leq\left(\sum_{n=0}^{\infty}\left(\E|Y_{n+1,i}-Y_{n,i}|^\alpha \right)^{1/\alpha}+1\right)^\alpha<\infty.$$
Notice that $\E\sum\limits_n|Y_{n+1,i}-Y_{n,i}|\leq\sum\limits_{n}\left(\E|Y_{n+1,i}-Y_{n,i}|^\alpha\right)^{1/\alpha}<\infty$. This fact leads to the a.s. convergence of the series $\sum\limits_n|Y_{n+1,i}-Y_{n,i}|$, which show that $\{Y_{n,i}\}$ is a Chauchy sequence in the sense a.s., so there exists a random variable $Y_i$ such that
$Y_{n,i}\rightarrow Y_i$ a.s.. By Fatou's Lemma, we have
\begin{eqnarray*}
\E|Y_{n,i}-Y_i|^\alpha=\E\lim_{l\rightarrow\infty}|Y_{n+l,i}-Y_{n,i}|^\alpha
\leq\liminf_{l\rightarrow\infty}\E|Y_{n+l,i}-Y_{n,i}|^\alpha
\leq\left(\sum_{k=n}^\infty(\E|Y_{k+1,i}-Y_{k,i}|^\alpha )^{1/\alpha}\right)^\alpha\stackrel{n\rightarrow\infty}{\longrightarrow }0.
\end{eqnarray*}
Thus $Y_{n,i}\rightarrow Y_i$ in $L^\alpha$, so that $\E Y_i=V_i$ and $0<\E(Y_i)^\alpha<\infty$.
\end{proof}

\end{document}